\documentclass[11pt]{amsart}
\usepackage{preamble}

\usepackage[headheight=15pt, headsep=15pt, footskip=27pt, bottom=2.4cm, left=2.4cm, right=2.4cm]{geometry}

\begin{document}

\title[Automorphism groups and Homogeneous fibrations]{Automorphism groups and Homogeneous fibrations on projective  varieties with nef anticanonical divisors}


\date{\today}

\begin{abstract}
We explore the homogeneous structures of Albanese morphisms and maximal rationally chain-connected fibrations of algebraic varieties with nef anticanonical divisors. Furthermore, we show that there exists a natural group homomorphism associated with a maximal rationally chain-connected fibration, which induces the Chevalley decomposition of $\Aut^0$ up to an isogeny. This result is analogous to a similar theorem of Nishi and Matsumura for Albanese morphisms. In addition, we reprove Nishi and Matsumura's theorem using the modern language of algebraic geometry.
\end{abstract}

\author{Zhan Li}
\address[Zhan Li]{Department of Mathematics, Southern University of Science and Technology, 1088 Xueyuan Rd, Shenzhen 518055, China} \email{lizhan@sustech.edu.cn, lizhan.math@gmail.com}

\author{Jinsong Xu}
\address[Jinsong Xu]{Department of Mathematical Sciences, Xi’an Jiaotong-Liverpool University, No.111 Renai Rd, Industrial Park, Suzhou 215000, China} \email{jinsong.xu@xjtlu.edu.cn}

\maketitle

\tableofcontents

\section{Introduction}\label{sec: intro}

Throughout this paper, we work over the field of complex numbers. 

Positivities of canonical bundles shape the geometry of algebraic varieties and the natural morphisms associated with them. When the canonical bundle of $X$ is trivial, it is shown by \cite[Theorem 1]{Cal57} that the Albanese morphism $a_X$ is locally trivial. This is generalized to varieties with mild singularities in \cite[Theorem 8.3]{Kaw85} (also see \cite{KL09}). This result is generalized to log pairs with trivial log canonical divisors by \cite[Theorem 0.1]{Amb05}. The arguments of \cite{Kaw85, Amb05} rely on deep results from the deformation theory and the Hodge theory. On the other hand, the second named author \cite[Theorem 1.1]{Xu20} generalizes \cite[Theorem 0.1]{Amb05} further and shows that the Albanese morphism $a_X$ can be identified with the natural $\Aut^0(X)$-equivariant morphism
\[
\Aut^0(X) \times^K F \to \Aut^0(X)/K,
\] where $K\subset \Aut^0(X)$ is a finite subgroup, $F$ is the fiber of $a_X$ over the identity element and $\Aut^0(X) \times^K F$ is the quotient of $\Aut^0(X) \times F$ under the natural action of $K$. This generalization explicitly reveals the homogeneous structure of the Albanese morphisms. Such a homogeneous structure plays a pivotal role in the study of the Morrison-Kawamata cone conjecture of Calabi-Yau fibrations (\cite{Kaw97} and \cite{Li23}). The argument of \cite[Theorem 1.1]{Xu20} differs completely from the previous work. It relies on the theory of algebraic groups, more precisely, a result of Brion (\cite[Theorem 2]{Bri10}, \cite[Proposition 1.2]{Bri12}) which generalizes a theorem of Nishi and Matsumura (\cite[Theorem 2]{Mat63}).

The local triviality of the Albanese morphism is also conjectured to hold for a K\"ahler manifold with nef anticanonical divisor (see \cite{DPS96}). This conjecture was first established in the projective setting by \cite{Cao19}. Subsequent works extended the result to log pairs with mild singularities and made further progress in the K\"ahler setting (see \cite{PZ19, Wan22, EIM23}), culminating in the proof of the conjecture in full generality by \cite{MWWZ25}. Additionally, the maximal rationally chain-connected fibrations (MRCC fibration) of such varieties exhibit analogous properties \cite{CH19, MW23}. The proof of these results relies on the positivity of the direct images of pluricanonical divisors. In view of \cite{Xu20}, it is still expected that the homogeneous structures can be attached to Albanese morphisms and the MRCC fibrations in this setting. As shown in this paper, this is indeed the case. 
 
 First, for Albanese morphisms, we have the following theorem.

\begin{theorem}\label{thm: homogeneous fibration of Albanese morphism}
Let $(X, \De)$ be a projective klt pair with nef log anticanonical divisor $-(K_X+\De)$. Then the Albanese morphism $a_X: X \to \Alb(X)$ can be identified with an $\Aut^0(X)$-equivariant fibration
\[
\psi: \Aut^0(X) \times^H F \to \Aut^0(X)/H,
\] where $F$ is the fiber of $a_X$ over the identity and $H \subset \Aut^0(X)$ is a normal subgroup such that $H/\Aut^0(X)_{\rm aff}$ is a finite group.
\end{theorem}

Please consult Section \ref{sec: pre} for the precise meaning of the notation. 

Second, for MRCC fibrations, we have the following theorem.

    \begin{theorem}\label{thm: product of MRCC}
        Let $(X, \De)$ be a projective klt pair with the nef log anticanonical divisor $-(K_X+\De)$. Then there exist a normal variety $X'$ and a finite quasi-\'etale cover $X' \to X$ satisfying the following properties:
        \begin{enumerate}
            \item $X'$ admits a locally constant MRCC fibration $\psi: X' \to Y$.
            \item $Y$ admits a locally constant Albanese morphism $a_Y: Y \to \Alb(Y) \simeq \Alb(X')$.
            \item Let $F$ and $S$ be the fibers of $a_Y\circ \psi$ and $a_Y$ over the identity of $\Alb(Y)$ respectively. Then, there exist normal subgroups 
            \[
            \Aut^0(X')_{\rm aff} \subset H \subset \Aut^0(X')
            \] with $H/\Aut^0(X')_{\rm aff}$ a finite group, and a finite subgroup  $K\subset \Aut^0(Y)$ such that $H$ is the preimage of $K$ under the natural group homomorphism $\Aut^0(X') \to \Aut^0(Y)$. Moreover, we have the following commutative diagram
             \[
\begin{tikzcd}
\Aut^0(X') \times^H F \arrow{r}{\simeq} \arrow{d}& X' \arrow{d}{\psi}\\
 \Aut^0(Y) \times^K S  \arrow{r}{\simeq}\arrow{d} & Y \arrow{d}{a_Y}\\
\Aut^0(X')/H \simeq \Aut^0(Y)/K \arrow{r}{\simeq} & \Alb(X') \simeq\Alb(Y),
\end{tikzcd}
\] where the morphisms in the left column are induced from the natural morphisms $\Aut^0(X') \to \Aut^0(Y)$ and $F \to S$. In particular, the MRCC fibration $\psi$ can be identified with the natural $\Aut^0(X')$-equivariant morphism
\[
\Aut^0(X') \times^H F \to \Aut^0(Y) \times^K S.
\] 
Besides, if $X$ is smooth and $\De=0$, then $X'$ can be taken to be $X$.
\end{enumerate}
    \end{theorem}

    The above two theorems are not independent of the analytic theorems of \cite{Cao19, CH19, MW23}; our argument builds upon the local constancy of Albanese morphism and MRCC fibrations in this setting. However, it indeed reveals the homogeneous structure of such fibrations, putting the previous analytic results in the algebraic categories. These algebraic formulations are anticipated to play a crucial role in the study of the Morrison-Kawamata cone conjecture and the moduli problem of varieties with nef anticanonical divisors.

    It is desirable to explore the relation between the automorphism group and the Albanese morphism for an arbitrary variety $X$. This is exactly the content of the theorem of Nishi and Matsumura (\cite[Theorem 2]{Mat63}, \cite[Theorem 5.5]{Fuj78}) which shows that there exists a natural group homomorphism 
    \[
    \Aut^0(X) \to \Alb(X)
    \] to the Albanese variety $\Alb(X)$ of $X$ which is the Chevalley decomposition of $\Aut^0(X)$ up to an isogeny (see Section \ref{sec: pre} for the precise meaning). We provide a proof of this theorem in the language of modern algebraic geometry. To be precise, we show the following result.
    
\begin{theorem}\label{thm: Chev for albanese}
    Let $X$ be a projective variety and $a_X: X \to A$ be the Albanese morphism. Let $G \subset \Aut^0(X)$ be a connected algebraic subgroup. Suppose that $X \to B$ is the Stein factorization of $a_X$. By Blanchard's lemma, there exists a natural group homomorphism
    \[
    \theta: G \to \Aut^0(B).
    \] Then 
    \begin{enumerate}
        \item $G \to {\rm Im}(\theta)$ is the Chevalley decomposition of $G$ up to an isogeny;
        \item the claim in (1) is equivalent to \cite[Theorem 2]{Mat63} (i.e., Theorem \ref{thm: Mat63}).
    \end{enumerate}
\end{theorem}

The above result is slightly different from Matsumura's original theorem in its formulation. In \cite[Theorem 2]{Mat63}, the group homomorphism involved is derived from the universal property of Albanese morphisms and the rigidity of morphisms between abelian varieties. However, then group homomorphism in Theorem \ref{thm: Chev for albanese} is obtained from Blanchard's lemma (Lemma \ref{lem: Blanchard}). We show that the two formulations are indeed equivalent. See Remark \ref{rmk: Mat63} for the relationship among Theorem \ref{thm: Chev for albanese}, Brion's version (\cite[Theorem 2]{Bri10}, \cite[Proposition 1.2]{Bri12}) and the above theorem.

Next, it is natural to proceed to study the relationship between automorphism groups and MRCC fibrations. A similar result as the theorem of Nishi and Matsumura still holds.

\begin{theorem}\label{thm: aut in MRCC fibration}
Let $X \dto Z$ be an MRCC fibration. Then, after a birational modification of $Z$, there exists a natural homomorphism of algebraic groups
\[
\theta: \Aut^0(X) \to \Aut^0(Z).
\] 
Moreover, if $G \subset \Aut^0(X)$ is a connected subgroup, then $G \to \theta(G)$ is the Chevalley decomposition of $G$ up to an isogeny. 
\end{theorem}

Note that an MRCC fibration $f: X \dto Z$ is an almost holomorphic projective map which means that $f$ is just a projective morphism restricting to an open set of $X$. Moreover, $f$ (and thus $Z$) is only well-defined up to birational equivalence. Thus, there is no natural group homomorphism $\Aut^0(X) \to \Aut^0(Z)$ as obtained by the classical Blanchard's lemma. To overcome this difficulty, we generalize Blanchard's lemma for almost holomorphic maps.

\begin{lemma}[Blanchard's lemma for almost holomorphic maps]\label{lem: Blanchard for almost hol maps}
    Let $f: X \dto Z$ be an almost holomorphic projective fibration between normal projective varieties. Then, there exist a variety $T$ that is birational to $Z$ and a natural homomorphism of algebraic groups
    \[
    \theta: \Aut^0(X) \to \Aut^0(T)
    \] which is compatible with Blanchard's lemma over the locus where $f$ is projective.
\end{lemma}

We discuss the contents of the paper. Section \ref{sec: pre} gives necessary background materials and fixes notation. In section \ref{sec: fibrations on varieties with nef anticanonical divisors}, we establish the homogeneous structures for Albanese morphisms and MRCC fibrations for varieties with nef log anticanonical divisors (i.e., Theorem \ref{thm: homogeneous fibration of Albanese morphism} and Theorem \ref{thm: product of MRCC}). Section \ref{sec: Albanese} is denoted to the proof of Theorem \ref{thm: Chev for albanese} (i.e., Theorem \ref{thm: Mat63}). Section \ref{sec: MRCC} is denoted to the proof of Theorem \ref{thm: aut in MRCC fibration} which is an analogous of Theorem \ref{thm: Mat63} for MRCC fibrations.

\medskip

\noindent {\bf Acknowledgments.} We benefit from discussions with Bo Yang and Juanyong Wang. We thank Sheng Meng for pointing out the possibilities of using the canonical constructions of MRCC fibrations (see Remark \ref{rmk: canonical choice of Z}). Zhan Li is partially supported by the NSFC No.12471041, the Guangdong Basic and Applied Basic Research Foundation No.2024A1515012341, and a grant from SUSTech. Jinsong Xu is partially supported by XJTLU-IRA No.00100000098.

\section{Preliminaries}\label{sec: pre}

We introduce the necessary background materials. Along with this process, we fix the notation and terminologies. 

A variety means an integral separated scheme of finite type over $\Cc$. A point of a variety is understood to be a closed point unless explicitly stated otherwise. A projective morphism $f: X \to S$ between normal varieties is called a fibration if it is surjective and $f_*\Oo_X=\Oo_S$. If $S' \to S$ is a morphism, then $X_{S'}$ denote the fiber product $X \times_{S'}S$. Suppose that $\De \geq 0$ is a $\Qq$-divisor on a normal variety $X$, then $(X,\De)$ is called a log pair. A log pair $(X,\De)$ has klt singularities if $K_X+\De$ is $\Qq$-Cartier where $K_X$ is the canonical divisor of $X$, and there exists a log resolution $\pi: Y \to X$ such that in the expression
\begin{equation}\label{eq: klt}
    K_Y=\pi^*(K_X+\De)+D,
\end{equation}
the coefficients of $D$ are greater than $-1$. Note that in \eqref{eq: klt}, $K_Y$ is chosen to be the unique Weil divisor on $Y$ such that $\pi_*K_Y=K_X$. Similarly, if the coefficients of $D$ are greater or equal to $-1$, then $(X,\De)$ is said to have lc singularities. See \cite[\S 2.3]{KM98} for more detailed discussions. 

For an algebraic variety $X$, we have the Albanese morphism $a_X: X \to \Alb(X)$ which is a morphism to an abelian variety satisfying the following universal property: for any morphism to an abelian variety $h: X \to B$, there exists a unique morphism $g: \Alb(X) \to B$ such that $h=g \circ a_X$. Similarly, we have the Albanese map $a^r_X: X \dto \Alb_r(X)$ which is a rational map to an abelian variety satisfying the analogous universal property for rational maps to abelian varieties. See \cite{Ser58} for details. The following proposition should be well-known.

\begin{proposition}\label{prop: albanese}
Let $X$ be a projective variety and $W\to X$ be a resolution.
\begin{enumerate}
\item If $X$ has rational singularities, then the natural morphism $\Alb(W) \to \Alb(X)$ is an isomorphism.
\item The Albanese map $a^r_X: X \dto \Alb_r(X)$ is exactly the natural map
\[
X \dto W \to \Alb(W).
\]
\end{enumerate}
\end{proposition}
\begin{proof}
The Albanese variety $\Alb(X)$ is the dual abelian variety of $\Pic^0(X)$. As $X$ has rational singularities, the natural morphism $\Pic^0(X) \to \Pic^0(W)$ is an isomorphism. This implies that the natural morphism $\Alb(W) \to \Alb(X)$ is an isomorphism. This shows (1).

For any map $X \dto B$ from $X$ to an abelian variety, after taking a resolution $W \to X$, we can assume that $W \to B$ is a morphism. Hence, there exists a natural morphism $\Alb(W) \to B$ by the universal property of $a_W$. As $W$ has rational singularities, by (1), we know that for any resolution $W' \to W$, the natural morphism $\Alb(W') \to \Alb(W)$ is an isomorphism. Therefore, $X \dto W \to \Alb(W)$ satisfies the universal property of the Albanese map. This shows (2).
\end{proof}

Let $f: X \dto Z$ be a rational map. Then $f$ is called almost holomorphic if there exists an open set $\emptyset \neq U \subset Z$ such that $f^{-1}(U) \to U$ is a morphism. Here, $f^{-1}(U)$ is defined as $p(q^{-1}(U))$ for some morphisms $p: W \to X$ and $q: W \to Z$ where $p$ is birational and $q=f\circ p$. An almost holomorphic map is a projective fibration if $f^{-1}(U) \to U$ is a projective fibration. Let $X$ be a normal projective variety, then there exist a normal variety $Z$ and an almost holomorphic projective fibration $f: X \dto Z$ such that fibers of the restricting projective fibration $f_U: f^{-1}(U) \to U$ are rationally chain-connected. Such an $f$ is called the maximal rationally chain-connected fibration (MRCC fibration) if it satisfies the universal property that for any almost holomorphic projective fibration $h: X \dto T$ with the aforementioned property, there exists a unique map $g: T \dto Z$ such that $f= g\circ h$. See \cite[\S 5.5]{Deb01} for the existence and properties of MRCC fibrations. We remark that when $X$ is a smooth projective variety over $\Cc$, then $X$ is rationally connected if and only if it is rationally chain-connected (see \cite[Corollary 4.28]{Deb01}). Let $\tilde{X} \to X$ be a resolution of $X$. If $X$ has klt singularities, then $X$ is rationally chain-connected if and only if $\tilde{X}$ is rationally chain-connected (see \cite[Corollary 1.6]{HM07}). Moreover, rational connectedness is preserved under birational equivalence for proper varieties \cite[(IV. (3.3.3))]{Kol96}.

 If \( X \) is a variety equipped with a regular action of an algebraic group \( G \), then \( X \) is called a \( G \)-variety. Let \( G \) be a connected algebraic group. We define \( A(G) \coloneqq \Alb(G) \) as the Albanese variety of \( G \). By a theorem of Chevalley, $G$ sits in an exact sequence of connected algebraic groups
\[
0 \to G_{\rm aff} \to G \to A(G) \to 0,
\] where $G_{\rm aff}$ is the maximal connected affine algebraic group. See \cite[\S 2]{BSU13} for details. We say an exact sequence of algebraic groups
\[
0 \to K \to G \to Q \to 0
\] (or just $G \to Q \to 0 $) is a Chevalley decomposition of $G$ up to an isogeny if $A(G)$ and $Q$ are isogeny abelian varieties. By the universal property of the Albanese morphism, this is equivalent to that $K/G_{\rm aff}$ is a finite group.

For an algebraic group $G$, we use $e$ to denote its identity element and $G^0$ to denote its identity component. For example, if $\Aut(X)$ is the algebraic group that represents the automorphism functor of $X$, then $\Aut^0(X)$ is its identity component.

We will use the Blanchard's lemma repetitively. The following version is taken from \cite[Proposition 4.2.1, Corollary 4.2.6 ]{BSU13}. See \cite[\S 4]{BSU13} for details.

\begin{lemma}[Blanchard's Lemma]\label{lem: Blanchard}
Let $f: X \to Y$ be a projective fibration. Let $G$ be a connected group scheme acting on $X$. Then there exists a unique $G$-action on $Y$ such that $f$ is $G$-equivariant. When $G = \Aut^0(X)$, then this action naturally induces a homomorphism of group schemes
\[
b_f: \Aut^0(X) \to \Aut^0(Y).
\]
\end{lemma}

Let $G$ be a connected algebraic group and $H\subset G$ be an algebraic subgroup. Suppose that $Y$ is an $H$-variety, then $H$ acts on $G \times Y$ by
\begin{equation}\label{eq: action}
  h \cdot (g, y)= (gh^{-1}, h y)  
\end{equation}
for $h\in H, g\in G, y\in Y$. If the quotient space is a variety, then we denote it by $G \times^H Y$. See \cite[\S 2.5]{Bri17} for detailed discussions.

We have the following definition of homogeneous fibrations (see \cite[Definition 1.1]{Bri12}).

\begin{definition}\label{def: homogeneous fibration}
A homogeneous fibration is a morphism $f: X \to A$ satisfying the following conditions:
\begin{enumerate}
\item $f_*\Oo_X = \Oo_A$,
\item $A$ is an abelian variety, and
\item $f$ is isomorphic to its pull-back by any translation in $A$.
\end{enumerate}
\end{definition}

The following result is a generalization of a theorem of Nishi and Matsumura (see \cite[Theorem 2]{Mat63}).

\begin{theorem}[{\cite[Theorem 2]{Bri10}, \cite[Proposition 1.2]{Bri12}}]\label{thm: group action}
Let $X$ be a normal projective variety and $G \subset \Aut^0(X)$ be a connected algebraic group.
\begin{enumerate}
\item For any homogeneous fibration $g: X \to A$, the $G$-action on $X$ induces an action of $A(G)$ on $A$ by translations.
\item There exists an $G$-equivariant homogeneous fibration $f: X \to A$ such that the resulting homomorphism $A(G) \to A$ is an isogeny.

\item For (2), we have $A = G/H$, where $H$ is a closed subgroup of $G$ such that $H \supset G_{\rm aff}$ and the quotient $H/G_{\rm aff}$ is finite. Then, we have $X \simeq G \times^H Y$, and $f$ can be identified with the natural $G$-equivariant fibration
\[
G \times^H Y \to G/H,
\]
where $Y$ is the scheme-theoretic fiber of $f$ over the identity element of $A$.
\end{enumerate}
\end{theorem}

\begin{remark}
The (3) in the above theorem is proved in \cite[\S 2.5, Page 97-98]{Bri17}.
\end{remark}

Let $\Bir(X)$ be the abstract group consisting of birational maps of $X$. Recall the notion of rational actions of algebraic groups on varieties.

\begin{definition}[{\cite[Definition 1, 2]{Kra18}}]\label{def: rational action}
Let $G$ be an algebraic group and $X$ be a variety. A map $\phi: G \to \Bir(X)$ is called a rational action if there is an open dense set $U \subset G \times X$ with the following properties:
\begin{enumerate}
\item The induced map $(g, x) \mapsto \phi(g)(x)$ from $U$ to $X$ is a morphism of varieties.
\item For every $g\in G$, the open set $U_g \coloneqq \{x \in X \mid (g, x) \in U\}$ is dense in $X$.
\item For every $g \in G$, the birational map $\phi(g): X \dto X$ is defined in $U_g$. 
\item $\phi: G \to \Bir(X)$ is a homomorphism of groups.
\end{enumerate}
\end{definition}

\begin{remark}\label{rmk: rational action}
(1) In the above definition, we also say that $X$ is equipped with the rational action of $G$.

(2) If $X$ is equipped with a rational action of $G$ and $X$ is an open subset of a variety $Y$, then $Y$ is naturally equipped with a rational action of $G$.

(3) On the other hand, if $X$ is equipped with a rational action of $G$, then for an open set $V \subset X$, $G$ may not act rationally on $V$. For one thing, there may not exist a well-defined induced morphism 
\[
U \cap (G \times V) \to V.
\] However, as a particular case of Lemma \ref{lem: rational action on open set}, if $G$ acts on $X$ regularly, then $G$ acts rationally on any open set of $X$.
\end{remark}

The following theorem is well-known to experts (for example, see \cite[\S 1.7, 1.8]{Kra18} and \cite[\S 1]{Bri22}). 

\begin{theorem}\label{thm: Weil's theorem}
Let $X$ be a variety equipped with a rational action of a connected algebraic group $G$. Then $X$ is $G$-equivariantly birational to a smooth projective variety $Y$ equipped with a regular $G$-action. 
\end{theorem}

\section{Homogeneous fibrations on varieties with nef anticanonical divisors}\label{sec: fibrations on varieties with nef anticanonical divisors}

As mentioned in the introduction, Albanese morphisms and MRCC fibrations of projective manifolds with nef anticanonical divisors are locally constant fibrations. These locally constant fibrations are transcendental. In this section, we establish similar results in the algebraic category by the theory of algebraic groups. Note that our proofs rely on these transcendental descriptions. Hence, they do not give new proofs of the aforementioned results.

First, recall the definition of locally constant fibration (see \cite[Definition 2.3]{MW23}).

\begin{definition}[locally constant fibration]\label{def: locally constant fibration}
Let $\phi: X \to S$ be a fibration between normal analytic varieties, and let $\De$ be a Weil $\Qq$-divisor on $X$. Then $\phi: X \to S$ is said to be a locally constant fibration with respect to the log pair $(X, \De)$ if it satisfies the following conditions:
\begin{enumerate}
    \item $\phi: X \to S$ is a locally trivial analytic fiber bundle with the fiber $F$.
    \item Every irreducible component $\De_i$ of $\De$ is horizontal (i.e., $\phi(\De_i)=S$).
    \item There exists a Weil $\Qq$-divisor $\De_F$ on $F$ and a representation $\rho: \pi_1(S) \to \Aut(F)$ of the fundamental group $\pi_1(S)$ to the automorphism group $\Aut(F)$ such that
    \begin{enumerate}
        \item $\De_F$ is invariant under the action of $\pi_1(S)$;
        \item $(X,\De)$ is isomorphic to the quotient $(\ti S \times^\rho F, {\rm pr}^*_2\De_F)$ over $S$. Here $\ti S$ is the universal cover of $S$ and $\gamma \in \pi_1(S)$ acts on $\ti S \times F$ by
        \[
        \gamma \cdot (s, z) \coloneqq (\gamma \cdot s, \rho(\gamma)(z))
        \] for any $(s,z) \in \ti S \times F$.
    \end{enumerate}
\end{enumerate}
\end{definition}

The following results generalize \cite{Cao19, CH19} from smooth projective varieties to projective varieties with klt singularities.

\begin{theorem}[{\cite[Theorem A]{Wan22}, \cite[Corollary 4.9]{MW23}}]\label{thm: locally constant fibration of Albanese morphism}
Let $(X, \De)$ be a projective klt pair with nef log anticanonical divisor $-(K_X+\De)$. Then, the Albanese map of $X$ is a locally constant fibration with respect to $(X,\De)$.
\end{theorem}

Recall that a finite morphism $X' \to X$ is called quasi-\'etale if it is \'etale outside a subset of codimension at least 2 in $X'$.

\begin{theorem}[{\cite[Theorem 1.4]{CH19}, \cite[Theorem 1.1]{MW23}}]\label{thm: locally constant fibration of MRCC}
Let $(X, \De)$ be a projective klt pair with the nef log anticanonical divisor $-(K_X+\De)$. Then, there exists a finite quasi-\'etale cover $\mu: X' \to X$ from a normal projective variety $X'$ satisfying the following properties:
\begin{enumerate}
    \item $X'$ admits a holomorphic MRCC fibration $\psi: X' \to Y$.
    \item $Y$ is a projective klt variety with a numerically trivial canonical divisor.
    \item $\psi: X' \to Y$ is a locally constant fibration with respect to $(X',\De')$, where $\De'$ is the Weil $\Qq$-divisor defined by the pullback $\De' \coloneqq \mu^*\De$.
\end{enumerate}
Moreover, if $X$ is smooth and $\De=0$, then $X'$ can be taken to be $X$.
    \end{theorem}
    
    \begin{remark}\label{rmk: indeed locally constant}
    In \cite[Theorem 1.4]{CH19}, they only assert that the MRCC fibration is locally trivial when $X$ is smooth and $-K_X$ is nef. However, their argument indeed shows that the MRCC fibration is locally constant. We thank Juanyong Wang for explaining this fact.
    \end{remark}

\subsection{Albanese morphisms of varieties with nef anticanonical divisors}\label{subsec: Albanese}

Under the notation of Definition \ref{def: locally constant fibration}, let $\pi: \ti S \to S$ be the universal cover of $S$. Let $T_X$ be the tangent sheaf of $X$ and $T_{X,x}$ be the tangent space at the point $x$.
  
\begin{lemma}\label{lem: lift vector field}
    If $f: X \coloneqq \ti S \times^\rho F \to S$ is a locally constant fibration, then the natural map 
    \[
    T_f: H^0(X, T_X) \to H^0(X,f^*T_S)=H^0(S, T_S)
    \] is surjective. 
\end{lemma}
\begin{proof}
    Choose $v \in H^0(S, T_S)$, then $v$ naturally lifts to a vector field $\ti v \in H^0(\ti S, T_{\ti S})$ such that
    \[
    \ti v(\gamma \cdot \ti t) = v(x) \in T_{\ti S, \gamma\ti t} \simeq T_{S, t}
    \] for any $\gamma\in \pi_1(S)$, $\ti t \in \ti S$ and $t = \pi(\ti t)$. Then, $\ti v$ trivially lifts to a vector field $\ti w$ on $\ti S \times F$ such that
    \[
    \ti w(\ti t, y)= \ti v(\ti t) \times \{0\} \in T_{\ti S, \ti t} \times T_{F,y},
    \] where $\ti t\in\ti S$ and $y\in F$.

    We claim that $\ti w$ is $\pi_1(S)$-invariant. That is,
    \[
    \ti w(\gamma \cdot \ti t, \rho(\gamma)(y)) = \ti w(\ti t, y)
    \] for any $\gamma\in \pi_1(S)$. This holds as 
    \[
    \ti w(\gamma\cdot \ti t, \rho(\gamma)(y)) = \ti v(\gamma\cdot \ti t) \times \{0\}= v(t) \times \{0\}
    \] and
    \[
    \ti w(\ti t, y) = \ti v(\ti t) \times \{0\}= v(t) \times \{0\}
    \] by construction.

    Therefore, $\ti w$ descents to a vector field $w$ on $X$. It is straightforward to see that $T_f(w)=v$.
\end{proof}

Recall that $A(\Aut^0(X))$ denotes the Albanese variety of the algebraic group $\Aut^0(X)$. The following is a direct consequence of Theorem \ref{thm: group action}. 

\begin{lemma}\label{lem: same as Alb}
    If $\dim A(\Aut^0(X)) \geq \dim \Alb(X)$, then there exists a normal subgroup $H \subset \Aut^0(X)$, with $\Aut^0(X)_{\rm aff} \subset H$ and $H/\Aut^0(X)_{\rm aff}$ a finite group, such that the Albanese morphism $a_X$ is the same as the natural morphism 
    \[
    \Aut^0(X) \times^H F \to \Aut^0(X)/H,
    \] where $F$ is the fiber of $a_X$ over the identity element of $\Alb(X)$.
\end{lemma}
\begin{proof}
    By Theorem \ref{thm: group action}, there exist a normal subgroup  $H \subset \Aut^0(X)$, with $\Aut^0(X)_{\rm aff} \subset H$ and $H/\Aut^0(X)_{\rm aff}$ a finite group, and an $\Aut^0(X)$-equivariant fibration
    \[
    f: X \simeq \Aut^0(X) \times^H Y \to \Aut^0(X)/H,
    \] where $Y$ is the fiber over the identity element of $\Aut^0(X)/H$. As $\Aut^0(X)/H$ is an abelian variety, by the universal property of Albanese morphisms, there exists a unique morphism $h: \Alb(X) \to \Aut^0(X)/H$ such that $f= h \circ a_X$. Hence, 
    \[
    \dim A(\Aut^0(X)) = \dim (\Aut^0(X)/H) \leq \dim \Alb(X).
    \] By the assumption, we see that $h$ is a generically finite morphism. As $f$ is a fibration, $h$ is also a fibration. Consequently, $h$ must be an isomorphism as it is a birational morphism between abelian varieties. 
\end{proof}

Now, we show that Albanese morphisms are homogeneous fibrations when log anticanonical divisors are nef.

\begin{proof}[Proof of Theorem \ref{thm: homogeneous fibration of Albanese morphism}]
By Theorem \ref{thm: locally constant fibration of Albanese morphism}, $a_X$ is a locally constant fibration. By Lemma \ref{lem: Blanchard}, there exists a natural homomorphism of algebraic groups
\[
b: \Aut^0(X) \to \Aut^0(\Alb(X)) \simeq \Alb(X).
\] 

Set $\Cc[\ep] = \Cc[x]/(x^2)$. Then, we have 
\[
T_{\Aut^0(X),e} = \Hom_\Cc(\spec \Cc[\ep], \Aut(X)) = \Aut(X_{\spec \Cc[\ep]}/\spec \Cc[\ep]) \simeq H^0(X, T_X)
\] (see \cite[Lemma 3.4]{MO67}). Furthermore, the natural morphisms yield the following commutative diagram
             \begin{equation}\label{eq: commutative diagram}
\begin{tikzcd}
T_{\Aut^0(X),e} \arrow{r}{\simeq} \arrow{d}{T_b}& H^0(X, T_X) \arrow{d}{T_{a_X}}\\
 T_{\Alb(X),e}   \arrow{r}{\simeq} & H^0(\Alb(X), T_{\Alb(X)}).
\end{tikzcd}
\end{equation} By Lemma \ref{lem: lift vector field}, the natural map
\[
T_{a_X}: H^0(X, T_X) \to H^0(X,a_X^*T_{\Alb(X)}) =H^0(\Alb(X), T_{\Alb(X)})
\] is surjective. Hence, $b: \Aut^0(X) \to \Aut^0(\Alb(X)) \simeq \Alb(X)$ is surjective. Therefore, by the universal property of the Albanese morphism $\Aut^0(X) \to A(\Aut^0(X))$, there exists a surjective morphisms $A(\Aut^0(X)) \to \Alb(X)$. In particular, we have
\[
\dim A(\Aut^0(X)) \geq \dim \Alb(X).
\] The desired claim then follows from Lemma \ref{lem: same as Alb}.
\end{proof}

\begin{remark}
  As shown by \cite[Theorem 1.1]{BFPT24}, Theorem \ref{thm: homogeneous fibration of Albanese morphism} fails even in the birational sense for lc singularities. To be precise, \cite[Theorem 1.1]{BFPT24} constructs a variety $X$ with lc singularities and $K_X \sim 0$. But any fiber of the Albanese morphism of $X$ is birational to at most finitely many other fibers.
\end{remark}

\subsection{MRCC fibrations of varieties with nef anticanonical divisors}\label{subsec: MRCC}

\begin{lemma}\label{lem: Alb for MRCC fibration}
Let $f:X \dto T$ be an almost holomorphic projective fibration between projective normal varieties with rational singularities. If there exists a non-empty open set $U\subset T$ such that $f^{-1}(U) \to U$ is a projective fibration whose fibers are rationally chain-connected. Then there exists a natural isomorphism $\Alb(X) \to \Alb(T)$.
\end{lemma}
\begin{proof}
    Let $X_U \coloneqq f^{-1}(U)$. By Proposition \ref{prop: albanese}, the natural morphisms $a_X|_{X_U}: X_U \to \Alb(X)$ and $a_T|_U: U \to \Alb(T)$ are Albanese maps for $X_U$ and $U$ respectively. Therefore, there exists unique morphism $h: \Alb(X) \to \Alb(T)$ such that $h \circ a_X|_{X_U} = a_T|_U\circ f|_{X_U}$. As fibers of $f|_{X_U}$ are rationally chain-connected, each fiber of $f|_{X_U}$ is contracted by $a_X|_{X_U}$. Therefore, there exists a morphism $g: U \to \Alb(X)$ such that $g \circ f|_{X_U} = a_X|_{X_U}$. Applying the universal property of the Albanese map $a_T|_{U}$, there exists a unique morphism $u: \Alb(T) \to \Alb(X)$ such that $u\circ a_T|_{U}=g$. By the uniqueness property, we have $u \circ h = \Id_{\Alb(X)}$ and $h \circ u = \Id_{\Alb(T)}$.
\end{proof}

We now establish the homogeneity of MRCC fibrations as stated in Theorem \ref{thm: product of MRCC}.

    \begin{proof}[Proof of Theorem \ref{thm: product of MRCC}] 
        By Theorem \ref{thm: locally constant fibration of MRCC}, there exist a normal projective variety $X'$ and a finite quasi-\'etale cover $\psi: X' \to X$ such that the MRCC fibration $X' \to Y$ is locally constant. Hence, each fiber of $\psi$ is a rationally connected variety. This implies $\Alb(Y) \simeq\Alb(X')$ by Lemma \ref{lem: Alb for MRCC fibration}. Moreover, $Y$ is a projective klt variety with $K_Y \equiv 0$ by Theorem \ref{thm: locally constant fibration of MRCC}. Thus, $a_Y$ is a locally constant fibration (see \cite{Xu20} or Theorem \ref{thm: locally constant fibration of Albanese morphism}). 
        
        Just as \eqref{eq: commutative diagram}, we have the following commutative diagrams.
             \[
\begin{tikzcd}
T_{\Aut^0(X'),e} \arrow{r}{\simeq} \arrow{d}& H^0(X', T_{X'}) \arrow{d}{T_{\psi}}\\
 T_{\Aut^0(Y),e}   \arrow{r}{\simeq} & H^0(Y, T_{Y})
\end{tikzcd} \quad
\begin{tikzcd}
T_{\Aut^0(Y),e} \arrow{r}{\simeq} \arrow{d}& H^0(Y, T_{Y}) \arrow{d}{T_{a_Y}}\\
 T_{\Alb(Y),e}   \arrow{r}{\simeq} & H^0(\Alb(Y), T_{\Alb(Y)}).
\end{tikzcd}
\] By Lemma \ref{lem: lift vector field}, $T_{\psi}$ and $T_{a_Y}$ are surjective maps. Therefore, the natural group homomorphisms 
        \[
        \Aut^0(X') \to \Aut^0(Y), \quad \Aut^0(Y) \to \Aut^0(\Alb(Y)) = \Alb(Y)
        \] are also surjective. 
        
        As $Y$ is not a uniruled variety, $\Aut^0(Y)$ is an abelian variety (for example, see \cite[Proposition 4.6]{Amb05} or \cite[Lemma 3.3]{Xu20}). By Lemma \ref{lem: same as Alb}, there exists a finite subgroup $K \subset \Aut^0(Y)$ such that $a_Y$ can be identified with the natural $\Aut^0(Y)$-equivariant morphism
        \[
        \Aut^0(Y) \times^K S \to \Aut^0(Y)/K,
        \] where $S$ is the fiber over the identity element of $\Alb(Y)$. Note that $K$ is the kernel of the natural group homomorphism $\Aut^0(Y) \to \Alb(Y)$. For the same reason, there exists a normal subgroup $H \subset \Aut^0(X')$, with $\Aut^0(X')_{\rm aff} \subset H $ and $H/\Aut^0(X')_{\rm aff}$ a finite group, such that the Albanese morphism $a_{X'} = a_Y \circ \psi$ can be identified with the natural $\Aut^0(X')$-equivariant morphism
        \[
        \Aut^0(X') \times^H F \to \Aut^0(X')/H,
        \] where $F$ is the fiber over the identity element of $\Alb(X')$. Here, $H$ is the kernel of the natural group homomorphism $\Aut^0(X') \to \Alb(X')=\Alb(Y)$. From the natural group homomorphisms
        \[
        \Aut^0(X') \to \Aut^0(Y) \to \Alb(X')=\Alb(Y),
        \] we see that $H$ is the preimage of $K$.

        Finally, by the constructions of $F$ and $S$, we have $\psi|_F: F \to S$. Therefore, there exists a natural $\Aut^0(X')$-equivariant morphism
        \[
        \varepsilon: \Aut^0(X') \times^H F \to \Aut^0(Y) \times^K S
        \] over $\Alb(Y)$. Moreover, $\varepsilon$ is the same as $\psi_F$ over the identity element of $\Alb(Y)$. On the other hand, as $\psi$ is an $\Aut^0(X')$-equivariant morphism under the same group homomorphism $\Aut^0(X') \to \Aut^0(Y)$, the fact that $\varepsilon$ and $\psi$ can be identified over the identity element of $\Alb(Y)$ implies that $\varepsilon$ and $\psi$ can be identified as the same $\Aut^0(X')$-equivariant morphism.

        When $X$ is smooth and $\De=0$, the last claim follows from Theorem \ref{thm: locally constant fibration of MRCC}.
    \end{proof}

\begin{remark}
    By the example in \cite[Remark 4.10]{MW23}, it is necessary to pass to a finite quasi-\'etale cover in order to obtain a locally constant fibration. 
    \end{remark}

    \begin{remark}
    \cite[Theorem 1.1]{BFPT24} also provides the example that Theorem \ref{thm: product of MRCC} fails even in the birational sense for lc singularities. To be precise, \cite[Theorem 1.1]{BFPT24} constructs a variety $X$ with lc singularities and $K_X \sim 0$. The Albanese morphism $a_X: X \to E$ is a fibration to an elliptic curve such that each fiber of $a_X$ is birational to at most finitely many other fibers. Moreover, a general fiber of $a_X$ is a rationally connected variety. This follows from the construction (see \cite[\S 4.2]{BFPT24}) and the fact that rational connectedness is preserved under birational equivalence \cite[(IV. (3.3.3))]{Kol96}). Therefore, $a_X$ is also an MRCC fibration of $X$. By \cite[Theorem 1.1 (3)]{BFPT24} , the natural homomorphism $\pi_1(X_{\rm reg}) \to \pi_1(E)$ is an isomorphism. Hence, if $Y \to X$ is a finite quasi-\'etale morphism, then it is induced from an \'etale base change $C \to E$ of $a_X$. Therefore, $C$ is an elliptic curve and $a_Y: Y \to C$ remains an MRCC fibration of $Y$. Moreover, each fiber of $a_Y$ is birational to at most finitely many other fibers, as this property also holds for $a_X$.
\end{remark}

\section{Albanese morphisms and Chevalley decompositions}\label{sec: Albanese}

This section aims to prove \ref{thm: Chev for albanese} which establishes a version of \cite[Theorem 2]{Mat63} using the framework of modern algebraic geometry over complex numbers.

Let $X$ be a projective variety with $G \subset \Aut^0(X)$ a connected algebraic group. Then, by the universal property of Albanese morphisms, we have a natural morphism $\Alb(G) \times \Alb(X) \to \Alb(X)$ such that the following diagram commutes
\begin{equation}\label{eq: commutative of alb}
    \begin{tikzcd}
G \times X \arrow{r} \arrow{d}{a_{G} \times a_X}& X \arrow{d}{a_X}\\
\Alb(G) \times \Alb(X) \arrow{r}{\beta}& \Alb(X),
\end{tikzcd}  
\end{equation}
where $G\times X \to X$ is the action of $G$ on $X$. One can check that there is a group homomorphism 
\begin{equation}\label{eq: group hom}
  \varphi: \Alb(G) \to \Alb(X)  
\end{equation}
such that
\begin{equation}\label{eq: construction of group hom}
    \beta(g,z) = \varphi(g)+z \text{~for~} g\in \Alb(G), z\in \Alb(X).
\end{equation} 
See \cite[\S 3]{Bri10} for details.

\begin{theorem}[{\cite[Theorem 2]{Mat63}}]\label{thm: Mat63}
Under the above notation, $\Ker \varphi$ is a finite group. In other words, $G \to \Alb(G) \to \Alb(X)$ is the Chevalley decomposition of $G$ up to an isogeny.   
\end{theorem}

 While Theorem \ref{thm: Chev for albanese} is slightly different from \cite[Theorem 2]{Mat63} (it is about the Stein factorization $B$, and the map comes from the Blanchard's lemma), we will show that this version is equivalent to \cite[Theorem 2]{Mat63}.

\begin{lemma}\label{lem: composition of Blanchard}
Let $p: X \to B$ and $q: B \to Y$ be fibrations between normal projective varieties. Let $b_p: \Aut^0(X) \to \Aut^0(B)$, $b_q: \Aut^0(B) \to \Aut^0(Y)$, and $b_{q\circ p}: \Aut^0(X) \to \Aut^0(Y)$ be the natural group homomorphisms induced by $p$, $q$, and $q \circ p$, respectively, as in Blanchard's lemma (see Lemma \ref{lem: Blanchard}). Then, we have $b_{q\circ p} = b_q \circ b_p$.
\end{lemma}
\begin{proof}
    Choose $g\in \Aut^0(X)$ and $y\in Y$. Set $g_B= b_p(g)$, $g_Y=b_{q\circ p}(g)$ and $h_Y=b_q(b_p(g))=b_q(g_B)$. Then, $g_Y(y) = (q\circ p)(g(X_y))$, where $X_y$ is the fiber of $q\circ p$ over $y$. On the other hand, $h_Y(y) = q(g_B(B_y))$, where $B_y$ is the fiber of $q$ over $y$. Let $b\in B_y$ be a point, then $g_B(b)=p(g(X_b))$, where $X_b$ is the fiber of $p$ over $b$. By construction, $X_b \subset X_y$, thus
    \[
    h_Y(y) = q(g_B(B_y))= q(p(g(X_b)))=(q\circ p)(g(X_y))=g_Y(y).
    \] This shows $h_Y = g_Y$. Therefore, we have $b_{q\circ p} = b_q \circ b_p$ as both sides are group homomorphisms of algebraic groups over $\Cc$.
\end{proof}

\begin{proof}[Proof of Theorem \ref{thm: Chev for albanese}]
By Theorem \ref{thm: group action}, there exist a normal subgroup $H \subset G$, with $G_{\rm aff} \subset H$ and $H/G_{\rm aff}$ a finite group, and a $G$-equivariant fibration
\[
\pi: X \simeq G \times^H F \to G/H,
\] where $F$ is the fiber over the identity element of $G/H$. As $G/H$ is an abelian variety, by the universal property of the Albanese morphism $a_X$, $\pi$ factors through $A$. As $X \xrightarrow{p} B \to A$ is the Stein factorization of $a_X$, there exists a morphism $q: B \to G/H$ such that $\pi= q \circ p$. As $\pi$ and $p$ are fibrations, $q$ is also a fibration. Applying Lemma \ref{lem: composition of Blanchard}, we have natural group homomorphisms
\[
b_\pi = b_q \circ b_p: \Aut^0(X) \to \Aut^0(B) \to \Aut^0(G/H)=G/H.
\] As $\theta = b_p|_{G}$ and $b_\pi (G) = G/H$, we have $\Ker(\theta) \subset H$. As $B$ is not a uniruled variety, $\Aut^0(B)$ is an abelian variety. Therefore, we have $G_{\rm aff} \subset \Ker(\theta)$. As $H/G_{\rm aff}$ is a finite group, $\Ker(\theta)/G_{\rm aff}$ is also a finite group. In other words, $G \to  {\rm Im}(\theta)$ is the Chevalley decomposition of $G$ up to an isogeny.

Next, we show that the above result is equivalent to \cite[Theorem 2]{Mat63} (i.e., Theorem \ref{thm: Mat63}). By the universal property of Albanese morphisms, $h: B \to A$ is also the Albanese morphism of $B$. As $\Aut^0(B)$ is an abelian variety, there exists a natural group homomorphism 
 \[
 \xi: \Aut^0(B) \simeq \Alb(\Aut^0(B)) \to A
 \] as explained in \eqref{eq: group hom}.
 
 \begin{claim}\label{claim: finite kernel}
  The kernel of $\xi: \Aut^0(B) \to A$ is finite.
 \end{claim}
 \begin{proof}
 First, we show that the isotropy group $\Aut^0(B)_b$ of $b\in B$ is a finite group. As $\Aut^0(B)$ is an abelian variety, by Theorem \ref{thm: group action}, there exist a finite group $K\subset \Aut^0(B)$ and a natural $\Aut^0(B)$-equivariant morphism
 \[
 B \simeq \Aut^0(B) \times^K F \to \Aut^0(B)/K.
 \] By the definition of $\Aut^0(B) \times^K F $, $(gt, x) \sim (t,x)$ in $\Aut^0(B) \times F $ if and only if there exists an $s\in K$ such that $(gts,s^{-1}x)=(t,x)$. As $\Aut^0(B)$ is abelian, this implies $gts=t$, and thus $g=s^{-1}\in K$. For $b=[(g, x)]$, this means
 \[
 |\Aut^0(B)_{[(g,x)]}| \leq |K|<\infty.
 \]

 Suppose that $\Ker(\xi)$ is infinite, then $\Ker(\xi)^0$ is a non-trivial algebraic group. Recall that $h: B \to A$ is a finite morphism which is the Albanese morphism of $B$. By \eqref{eq: commutative of alb} and \eqref{eq: construction of group hom}, we have
 \[
 h(g\cdot b)=\xi(g)+h(b)=h(b)
 \] for $g\in \Ker(\xi)^0$ and $b\in B$. Thus, 
 \[
 \Ker(\xi)^0 \cdot b\subset h^{-1}(h(b))
 \] is a finite set. This contradicts the fact that $\Aut^0(B)_b$ is a finite set.
 \end{proof}

  We note that the following two group homomorphisms are the same: 
 \begin{equation}\label{eq: a}
    A(G) \xrightarrow{\ti \theta} \Aut^0(B) \xrightarrow{\xi} A, 
 \end{equation}
 where $\ti\theta$ is induced from $\theta: G \to \Aut^0(B)$, and
 \begin{equation}\label{eq: b}
 \nu: A(G) \to A,
 \end{equation}
 which is obtained as in \eqref{eq: group hom}. Indeed, \eqref{eq: a} is derived from the following commutative diagram 
 \[
\begin{tikzcd}
    G \times X \arrow{r}\arrow{d} & X\arrow{d}\\
    \Aut^0(B) \times B \arrow{r} \arrow{d}& B \arrow{d}\\
    \Aut^0(B) \times A \arrow{r}{\beta_1}& A,
\end{tikzcd}
 \] and \eqref{eq: b} is derived from the following commutative diagram 
  \[
\begin{tikzcd}
    G \times X \arrow{r}\arrow{d} & X\arrow{d}\\
    A(G) \times A \arrow{r}{\beta_2}& A.
\end{tikzcd}
 \] Let $\ti g =a_G(g)\in A(G)$. Hence, for $\ti g\in A(G)$ and $x\in X$, we have
 \[
 \beta_1\left((\ti\theta(\ti g), a_X(x))\right)=a_X(g\cdot x)=\beta_2\left((\ti g,  a_X(x))\right).
 \]  By \eqref{eq: construction of group hom}, we have
 \[
 \begin{split}
 &\beta_1\left((\ti\theta(\ti g), a_X(x))\right)=(\xi\circ\ti\theta)(\ti g)+a_X(x),\\
 &\beta_2\left((\ti g,  a_X(x))\right) = \nu(\ti g)+a_X(x).
 \end{split}
 \] This shows $\xi\circ\ti\theta = \nu$. By Claim \ref{claim: finite kernel}, we have
 \[
 [\Ker(\xi \circ \ti\theta)/G_{\rm aff}: \Ker(\ti\theta)/G_{\rm aff}] \leq |\Ker(\xi)|<\infty.
 \] This shows (2) in the theorem.
\end{proof}

\begin{remark}
The ``up to an isogeny" in the theorem is necessary. For example, consider a hyper-elliptic surface $X \coloneqq E \times F/\langle 1, \sigma \rangle$ where $E, F$ are elliptic curves and $\sigma \in E$ is a non-trivial torsion element of order $2$. Let $\sigma$ act on $E$ by translations and on $F$ by sending $y$ to $-y$. Then, $E= \Aut^0(X)$ acts naturally on the first factor of $E \times F/\langle 1, \sigma \rangle$. The natural projection
\[
E \times F/\langle 1, \sigma \rangle \to E/\langle 1, \sigma \rangle
\] is the Albanese morphism. Hence, the group homomorphism $\theta: \Aut^0(X) \to \Alb(X)$ is exactly the quotient map $E \to E/\langle 1, \sigma \rangle$. It is not the Chevalley decomposition of $E$ but a Chevalley decomposition of $E$ up to an isogeny.
\end{remark}

\begin{remark}\label{rmk: Mat63}
We have a list of remarks regarding this theorem.
\begin{enumerate}
    \item \cite[Theorem 2]{Mat63} also asserts that the theorem holds in positive characteristics. 

    \item Given \cite[Theorem 2]{Mat63}, there exists a sub-abelian variety $G/H \subset \Alb(X)$ isogeny to $G/G_{\rm aff}$. By Poincar\'e's complete reducibility theorem, there exists an isogeny $\Alb(X) \to G/G_{\rm aff} \times B$ with $B$ an abelian variety. Therefore, there exists a natural $G$-equivariant morphism $X \to G/G_{\rm aff}$. See discussions in \cite[\S 4, page 12]{Bri10}. \cite[Theorem 2]{Bri10} (i.e., Theorem \ref{thm: group action}) reproves this result in the language of modern algebraic geometry. However, \cite[Theorem 2]{Bri10} does not recover the original result \cite[Theorem 2]{Mat63}.

    \item On the other hand, in view of (2), our proof does not give a new proof of \cite[Theorem 2]{Bri10} as we use \cite[Theorem 2]{Bri10}.
    \end{enumerate}
\end{remark}

\section{Maximal rationally chain-connected fibrations  and Chevalley decompositions}\label{sec: MRCC}

This section aims to prove Theorem \ref{thm: aut in MRCC fibration}, which is an analogy of \cite[Theorem 2]{Mat63} and Theorem \ref{thm: Chev for albanese} for MRCC fibrations.

\begin{lemma}\label{lem: rational action on open set}
Let $X$ be a variety equipped with a regular action of an algebraic group $G$. Suppose that $U\subset X$ is an open set which is an $L$-variety for a normal algebraic subgroup $L\subset G$. Assume that $U/L$ exists as a good geometric quotient and the quotient morphism $q: U \to U/L$ is flat. Then $G$ acts rationally on $U/L$ through
\[
G \times U/L \dto U/L \quad (g, [Lx]) \mapsto [Lgx]. 
\]
Moreover, $G/L$ acts rationally on $U/L$ through
\[
G/L \times U/L \dto U/L \quad ([Lg], [Lx]) \mapsto [Lgx]. 
\]
In particular, for any open set $V \subset X$, $G$ acts rationally on $V$.
\end{lemma}
\begin{proof}
Consider the set
\[
W \coloneqq \{(g, g^{-1}U) \mid g\in G\} \subset G \times X.
\] Note that there exists an isomorphism
\[
G \times X \to G \times X: (g, x) \mapsto (g, g^{-1}x),
\] and $W$ is the image of $G \times U$ under this isomorphism. Hence, $W$ is an open set. Therefore, $V \coloneqq (G \times U) \cap W$ is an open set. As a flat morphism is universally open, $q_G: G \times U \to G \times U/L$ is an open map. Thus
\[
\ti V \coloneqq q_G(V)
\] is open in $G \times U/L$. Moreover, over $g\in G$, we have
\[
\ti V_g = (g, q(U \cap g^{-1} U)) \subset (g, U/L).
\]

Hence, for any $(g, q(x)) \in \ti V$ with $x\in U\cap g^{-1}U$, the natural action of $G$ on $X$ induces $g\cdot q(x)=q(g\cdot x) \in U/L$. This verifies  (1) in Definition \ref{def: rational action}. The remaining conditions of a rational action are straightforward to check.

As for the $G/L$-action, the same argument works with $\ti V$ replaced by $q_{G/L}(V)$, where $q_{G/L}=1_{G/L} \times q: G/L \times U \to G/L \times U/L$.
\end{proof}

\begin{lemma}\label{lem: action on RCC variety}
    Let $X$ be a normal projective variety that is rationally chain-connected. If a connected algebraic group $G$ acts on $X$ faithfully, then $G$ is a linear algebraic group.
\end{lemma}
    \begin{proof}
        By Theorem \ref{thm: group action}, there exists a fibration $X \to A$, where $A$ is isogeny to $\Alb(G)$. If $X$ is rationally chain-connected, then so is $A$. Since $A$ is an abelian variety, $A$ must be trivial. Thus, $\Alb(G)$ is also trivial. By the Chevalley decomposition of $G$, $G= G_{\rm aff}$ is a linear algebraic group.
    \end{proof}

The following result is the key ingredient to generalize Blanchard's lemma for almost holomorphic projective fibrations.

\begin{theorem}\label{theorem: rational action}
    Let $f: X \dto Z$ be an almost holomorphic projective fibration between normal projective varieties. Then there exists a natural rational action of $\Aut^0(X)$ on an open subset $V \subset Z$. 
    
    To be precise, suppose that $f_W: X_W \to W$ is a projective fibration for some open set $W \subset Z$ and $X_W=f^{-1}(W)$. Then, there exists an open set $V \subset W$ satisfying the property that if $g\in \Aut^0(X)$ and $X_z$ is the fiber over some $z\in V$ such that $g \cdot X_z \subset X_W$, then $g \cdot X_z = X_{g \cdot z}$ for some $g\cdot z \in W$.    
\end{theorem}
\begin{proof}
Let $G \coloneqq \Aut^0(X)$. Take an open set $V \subset W$ such that $X_{V} \to V$ is flat and each closed fiber over $V$ is normal. Consider the set 
\[
T \coloneqq \{(g, g^{-1}\cdot X_{V}) \mid g\in G\} \subset  G \times X.
\] As shown in Lemma \ref{lem: rational action on open set}, $T$ is an open set. Thus,
\[
U' \coloneqq (G \times X_{V}) \cap T
\] is still an open set such that $U'_g = (g, X_{V} \cap g^{-1}\cdot X_{V})$. Moreover, if $X_t \cap (X_{V} \cap g^{-1}\cdot X_{V}) \neq \emptyset$ for some $t\in V$, then 
\begin{equation}\label{eq: intersects non-empty}
(g\cdot X_t) \cap X_{V} \neq \emptyset.
\end{equation}   

\medskip

Step 1. First, we show that if $X_t \cap (X_{V} \cap g^{-1}\cdot X_{V}) \neq \emptyset$, then $f((g\cdot X_t) \cap X_{V})$ is a point (by \eqref{eq: intersects non-empty}, this makes sense). As $G$ is quasi-projective, there exists an irreducible curve $\gamma \subset G$ (may not be projective) such that $e, g \in \gamma$, where $e\in G$ is the identity element. Consider the following diagram, with the morphisms explained below.

  \[
    \begin{tikzcd}
    &&T \arrow{lld}[swap]{p}\arrow{d} \arrow{rd}{q}&\\
    \gamma \times X_t \arrow[hook]{r}\arrow{d}{{\rm pr}_1} & G \times X \arrow{r}& X \arrow[dashed]{d}{f}& B \arrow{ld}{\tau}\\
    \gamma \arrow{rr}{h}&&Z &
    \end{tikzcd}
    \] In the above diagram, $G \times X \to X$ is the regular action of $G$ on $X$. Let $T$ be a normal variety such that $T \to \gamma \times X$ and $T \to Z$ are morphisms such that the diagram is commutative (ignore $h, q$ at the moment). Moreover, the $T$ can be chosen such that $p$ is birational. To be precise, one first takes a $T'$ that resolves $f$, then as $f$ is defined on $X_t$, there exists a map $\gamma \times X_t \dto T'$. The $T$ can be taken to be the resolution of this map. In particular, $p$ is a fibration as $\gamma$ and $X_t$ are normal varieties. Let 
    \[
    T \xrightarrow{q} B \xrightarrow{\tau} Z
    \] be the Stein factorization of $\tau\circ q$. Note that $T \to X$ and $T \to Z$ may not be surjective.

    As $e \in \gamma$ and $X_t \subset X_V$, by the commutative of the diagram, we see that $q(T_e)$ is a point, where $T_e$ is the fiber of ${\rm pr}_1 \circ p$ over $e$ (this uses the fact that $T_e$ maps to a point in $Z$ and $T_e$ is connected). By \cite[Lemma 1.15 (a)]{Deb01}, there exists a Zariski open set $e\in O \subset \gamma$ and a morphism $O \to B$ making the diagram commutes. As $\gamma$ is a curve, this morphism extends to a morphism $\gamma \to B$. Composing with $B \to Z$, we have $h$. Therefore, we have
    \[
    \tau\circ q(T_g)=h(g),
    \] where $T_g$ is the fiber of ${\rm pr}_1\circ  p$ over $g\in \gamma$. 
    This implies that $f((g\cdot X_t) \cap X_{V})=h(g)$.

    \medskip

Step 2. In this step, we construct the open set $U \subset G \times V$ for the rational action of $G$ on $V$. Recall that in Step 1, we constructed an open set $U' \subset G \times X_{V}$ and $X_{V} \to V$ is a flat morphism. Thus, $\iota: G \times X_{V} \to G \times V$ is also flat by \cite[\href{https://stacks.math.columbia.edu/tag/01UA}{Lemma 01UA}]{stacks-project}. In particular, $\iota$ is an open map. Let
\[
U \coloneqq \iota(U') \subset G \times V,
\] then $U$ is an open subset. Moreover, $U_g= f (X_{V} \cap g^{-1}\cdot X_{V})$, hence $U_g$ is a non-empty open set of $V$ (as $X_{V} \cap g^{-1}\cdot X_{V}$ is a non-empty open set of $X_{V}$). Take any $(g, t)\in U$, we have $t\in U_g$ and thus $X_t \cap (X_{V} \cap g^{-1}\cdot X_{V}) \neq\emptyset$. By Step 1, $f((g\cdot X_t) \cap X_{V})$ is a point in $V$ which is denoted by $g\cdot t$. Thus, this gives a map
\begin{equation}\label{eq: action on set}
\varphi: U \to V, \quad (g, t) \mapsto g\cdot t.
\end{equation} 

Step 3. We show that \eqref{eq: action on set} is a morphism (i.e., (1) in Definition \ref{def: rational action}). We follow the process in the proof of \cite[Proposition 4.2.1]{BSU13}.

First, we show that $\varphi$ is continuous. We have the following commutative diagram.
\[
\begin{tikzcd}
U' \arrow{r}{\psi}\arrow{d}{\iota}& X_{V} \arrow{d}{f}\\
U \arrow{r}{\varphi}&V,
\end{tikzcd}
\] where $\psi$ is the rational action of $G$ on $X_{V}$ (see Lemma \ref{lem: rational action on open set}). Therefore, if $\Omega\subset V$ is an open set, then
\[
\varphi^{-1}(\Omega)= \iota\left(\psi^{-1}(f^{-1}(\Omega))\right)
\] is an open set as $\iota$ is an open map. 

Next, we show that there exists a natural morphism of sheaves of $\Oo_{V}$-modules
\[
\varphi^*: \Oo_{V} \to \varphi_*\Oo_{U}.
\] Take $\alpha(z) \in \Oo_V(\Omega)$, we have $f^*\alpha(z)=\alpha(f(x)) \in \Oo_X(X_\Omega)$. Thus
\[
\psi^*(f^*\alpha)(g,x)=\alpha(f(g\cdot x)) \in \Oo_{G \times X}(\psi^{-1}(X_{V})).
\] For any $(g, x) \in X_t$, it is shown in Step 1 that $f(g\cdot x) = g\cdot t$. Thus, $\alpha(f(g\cdot x))$ is a constant on $(g, X_t)$. This implies that $\alpha(f(g\cdot x))$ is the pull-back of a holomorphic function on $\varphi^{-1}(\Omega)$. As  
\[
\alpha(f(g\cdot x))=  \alpha(\varphi(\iota((g, x))))=\iota^*(\varphi^*(\alpha))(g,x),
\] we see that $\varphi^*(\alpha) \in \Oo_{G \times V}(\varphi^{-1}(\Omega))$. From this, it is straightforward to check that $\varphi^*$ is a morphism of sheaves of $\Oo_{V}$-modules. Hence, we have shown (1) in Definition \ref{def: rational action}. The remaining conditions in Definition \ref{def: rational action} are straightforward to check from the construction.
 \end{proof}

Now, we show the Blanchard's lemma for almost holomorphic projective fibrations.

\begin{proof}[Proof of Lemma \ref{lem: Blanchard for almost hol maps}]
By Theorem \ref{theorem: rational action}, there exists an open set $V \subset Z$ that equips with a rational action of $\Aut^0(X)$. By Theorem \ref{thm: Weil's theorem}, there exists a birational modification $T \dto V$, where $T$ is a smooth projective variety such that $\Aut^0(X)$ acts regularly on $T$. Therefore, by the functorial property of $ \Aut(T)$, we have a morphism of varieties
\[
\theta: \Aut^0(X) \to \Aut^0(T).
\] By the construction of $\theta$, it is straightforward to check that $\theta$ is a group homomorphism restricted to the closed points $ \Aut^0(X)(\Cc) \to \Aut^0(T)(\Cc)$. As the base field is $\Cc$, this shows that $\theta$ is a group homomorphism. Note that $X \dto T$ is still an MRCC fibration. It is compatible with Blanchard's lemma over the locus where $f$ is projective by the explicit description of group actions in Theorem \ref{theorem: rational action}.
\end{proof}

\begin{remark}\label{rmk: canonical choice of Z}
    It is possible to use Nakayama's Chow reduction (see \cite[Definition 4.15]{Nak10}) to construct a canonical $T$. Indeed, \cite[Theorem 4.18, Theorem 4.19]{Nak10} uses this to construct a special MRC fibration (but this MRC fibration is not the same as the MRCC fibration considered in this paper, see \cite[Remark after Theorem 4.18]{Nak10}). For the sake of simplicity and the explicit description of group actions, we stick to the above non-canonical construction of $T$.
\end{remark}

We require the following technical lemma before proceeding to the proof of Theorem \ref{thm: aut in MRCC fibration}.

\begin{lemma}\label{lem: faithful on general fiber}
Let $f: X \dto T$ be an almost holomorphic projective fibration which is a projective morphism over $U\subset T$. Let $K \subset \Aut(X)$ be a connected algebraic subgroup such that each fiber $X_t, t\in U$ is $K$-invariant. Then, $K$ acts faithfully on general fibers over $U$.
\end{lemma}
\begin{proof}
By assumption, $K$ can be identified with a connected subgroup of $\Aut(X_V/V)$ for any open subset $V\subset U$. Hence, it can also be identified with a connected subgroup of $\Aut(X_\eta)$, where $\eta$ is the generic point of $U$. By abuse of notation, we still use $K$ to denote the corresponding identifications. By the representability of automorphism group functor, we have $\Aut(X_U/U)_\eta = \Aut(X_\eta)$, and thus $\Aut^0(X_U/U)_\eta \supset \Aut^0(X_\eta)$, where $\Aut^0(X_U/U)$ denotes the identity component of $\Aut(X_U/U)$. Therefore, there exists an open subset $V\subset U$ such that $\Aut^0(X_{V'}/V')_\eta =\Aut^0(X_\eta)$ for any open subset $V'\subset V$ (see \cite[\href{https://stacks.math.columbia.edu/tag/055H}{Lemma 055H}]{stacks-project}). Replacing $U$ by $V$, we can assume that $\Aut^0(X_{V'}/V')_\eta =\Aut^0(X_\eta)$ for any open subset $V'\subset U$. As $K \subset \Aut^0(X_U/U)$, it suffices to show the claim for $\Aut^0(X_U/U)$. In the sequel, we set $K=\Aut^0(X_U/U)$. Hence, we have $K_\eta = \Aut^0(X_\eta)$.  

Let $\bar\eta$ be the geometric generic point of $U$. As $\Aut^0(X_\eta)_{\bar\eta}=\Aut^0(X_{\bar\eta})$ by \cite[Lemma 9.5.1 (3)]{Kle05}, we have $K_{\bar\eta}=\Aut^0(X_{\bar\eta})$. In particular, $K_{\bar\eta}$ acts faithfully on $X_{\bar\eta}$. Let $K_{\bar\eta, \bar x}$ be the stabilizer of the closed point $\bar x \in X_{\bar\eta}$. As 
\[
\bigcap_{\bar x\in X_{\bar\eta}} K_{\bar\eta, \bar x} = \{e\},
\] by the noetherian property of algebraic varieties, there are finitely many closed points $\bar x_1, \ldots, \bar x_m \in X_{\bar\eta}$ such that 
\[
\bigcap_{i=1}^m K_{\bar\eta, \bar x_i} = \{e\}.
\] Hence, there exist an open subset $V\subset U$, a finite base change $V' \to V$, and sections $S_i \subset X_{V'}, i =1, \ldots, m$ over $V'$ such that $S_{i,\bar\eta} = \bar x_i$. By the choice of $U$, we have $K_V = \Aut^0(X_V/V)$, where $K_V$ denotes the base change of $K$ through $V \to U$. As $\Aut^0(X_\eta)_{\bar\eta}=\Aut^0(X_{\bar\eta})$, after shrinking $V'$ and $V$, we have $K_{V'} = \Aut^0(X_{V'}/V')$. 

Let $S \coloneqq \cup_{i=1}^m S_i$ be the subscheme of $X_{V'}$ with reduced induced structure. Let $\Aut(X_V'/V', S)$ be the scheme that represents the automorphism group subfunctor of $X_{V'}$ over $V'$ which leaves $S$ invariant (see \cite[\S 2]{Xu20}). Then, for each $t\in V'$, we have
\begin{equation}\label{eq: invariant auto}
\Aut(X_{V'}/V', S)|_{X_t} \simeq \Aut(X_t, S|_{X_t}).
\end{equation} Let $K_{V', S}$ be the stabilizer of the subscheme $S$. 

\begin{claim}\label{claim: restriction}
We have
\[
K_{V', S}|_{X_t} = K_{V', s},
\] where $s= S|_{X_t}$ and $K_{V', s}$ is the stabilizer of $s$ for $K_{V'}$ acting on $X_t$.
\end{claim}
\begin{proof}[{Proof of Claim \ref{claim: restriction}}]
First, we have the natural inclusion $K_{V', S}|_{X_t} \subset K_{V', s}$. By \eqref{eq: invariant auto}, we have $K_{V', s}  \subset \Aut(X_{V'}/V', S)|_{X_t}$. Moreover, we have
\[
K_{V', s}  \subset K_{V'}|_{X_t} = \Aut^0(X_{V'}/V')|_{X_t} .
\] Therefore, we have
\[
K_{V', s}  \subset \left(\Aut^0(X_{V'}/V') \cap \Aut(X_{V'}/V', S)\right)|_{X_t}. 
\] As $K_{V', S} = \Aut^0(X_{V'}/V') \cap \Aut(X_{V'}/V', S)$, we have $K_{V', s} \subset K_{V', S}|_{X_t}$. This shows the claim.
\end{proof}

Finally, we have   
\[
\{e\} \subset K_{V', S}|_{\bar\eta} \subset \bigcap_{i=1}^m K_{\bar\eta, \bar x_i} = \{e\}.
\]  
Using Claim \ref{claim: restriction}, we conclude that  
\[
K_{V', s} = K_{V', S}|_{X_t} = \{e\}
\]  
for a general \( t \in V' \). In other words, $K_{V'}$ acts faithfully on the fiber $X_t$. As $X_{V'} \to X_V$ is obtained through the base change $V' \to V$, $K$ also acts faithfully on general fibers over $V$.
\end{proof}

Now, we show Theorem \ref{thm: aut in MRCC fibration} which reveals the relationship between Chevalley decompositions and MRCC fibrations.

\begin{proof}[Proof of Theorem \ref{thm: aut in MRCC fibration}]
First, by Lemma \ref{lem: Blanchard for almost hol maps}, Then there exist a variety $T$ that is birational to $Z$ and a natural homomorphism of algebraic groups
    \[
    \theta: \Aut^0(X) \to \Aut^0(T)
    \] which is compatible with Blanchard's lemma over the locus where $f$ is projective. Replacing $Z$ by $T$, we obtain the desired group homomorphism. 

Next, we show that $G \to \theta(G)$ is the Chevalley decomposition of $G$ up to an isogeny. As $T$ is not a uniruled variety (see \cite[Corollary 1.4]{GHS03}), $\Aut^0(T)$ is an abelian variety. Thus, $\theta(G)$ is also an abelian variety. Let $H = (\Ker \theta)^0$. By the construction of $\theta$, $\Ker \theta$ acts on general fibers $X_z, z\in T$. By Lemma \ref{lem: faithful on general fiber}, $H$ acts on $X_z$ faithfully for general $z\in T$. By the definition of MRCC fibrations, a general fiber $X_z$ is rationally chain-connected. Moreover, it is a normal variety as $X$ is normal. By Lemma \ref{lem: action on RCC variety}, $H$ is a linear algebraic group. Hence, if $\beta: G \to A(G)$ is the Chevalley decomposition of $G$, then $H \subset \Ker(\beta)$. Therefore, we have
\[
\Ker \theta/ \Ker \beta \simeq  \frac{\Ker \theta/ H}{\Ker \beta/ H}.
\] As $\Ker \theta/ H=\Ker \theta/ (\Ker \theta)^0$ is a finite group, $\Ker \theta/ \Ker \beta$ is a finite group. This shows the desired claim.
\end{proof}

\begin{remark}
The ``up to an isogeny" in the theorem is necessary. Consider the variety $X = E \times^K \Pp^1$ where $E$ is an elliptic curve and $K=\langle 1, \sigma \rangle\subset E$ is the subgroup generated by a torsion element $\sigma$ of order $2$. Suppose that $\sigma$ acts on $E$ by translation and on $\Pp^1$ by $x \mapsto -x$. Then, the natural morphism $E \times^K \Pp^1 \to E/K$ is the MRCC fibration (it is also the Albanese morphism). Note that $E = \Aut^0(X)$ which acts on the first factor of $E \times^K \Pp^1$. Hence $\theta: \Aut^0(X) \to \Aut^0(E/K)$ is exactly the quotient morphism 
\[
E \to E/K.
\] This is not the Chevalley decomposition of $\Aut^0(X)$ but an isogeny to it.
\end{remark}

\begin{example}
The following examples on ruled surfaces unify the discussions in the paper. 

Let $E$ be an elliptic curve and $\Ee$ be a rank $2$ vector bundle on $E$. Then, $f: X \coloneqq \Pp_E(\Ee) \to E$ is both the Albanese morphism and the MRCC fibration of $X$. Let $b: \Aut^0(X) \to \Aut^0(E)=E$ be the natural group homomorphism in Blanchard's lemma.

(1) Let $\Ee = \Oo_E \oplus \Ll$, where $\Ll$ is a line bundle on $E$ with $\deg \Ll =0$. Then, we have the exact sequence 
\[
0 \to {\mathbb G}_m \to \Aut^0(X) \xrightarrow{b} E \to 0,
\] where ${\mathbb G}_m$ is the multiplicative group $\Cc^*$. Therefore, we have $A(\Aut^0(E))=E$ and $\Aut^0(X)_{\rm aff}={\mathbb G}_m$. Besides, we have $\Oo_X(K_X)=f^*\det \Ee \otimes \Oo_X(-2)$, and thus $-K_X$ is nef. The locally trivial fibration $f$ is a locally constant fibration, and it splits to a product after a finite \'etale base change if and only if $\Ll$ is a torsion line bundle.

(2) Let $\Ee$ be an indecomposable vector bundle which sits in an exact sequence 
\[
0 \to \Oo_E \to \Ee \to \Oo_E \to 0.
\]Then, we have the exact sequence 
\[
0 \to {\mathbb G}_a \to \Aut^0(X) \xrightarrow{b} E \to 0,
\] where ${\mathbb G}_a$ is the additive group $\Cc$. Therefore, we have $A(\Aut^0(E))=E$ and $\Aut^0(X)_{\rm aff}={\mathbb G}_a$. Besides, we have $\Oo_X(K_X)=f^*\det \Ee \otimes \Oo_X(-2)$, and thus $-K_X$ is nef. The locally trivial fibration $f$ is a locally constant fibration. However, it does not split to a product after any finite \'etale base change.

(3) Let $\Ee = \Oo_E \oplus \Ll$, where $\Ll$ is a line bundle on $E$ with $\deg \Ll <0$. Then, $X \to E$ admits a section $C_0$ with $C^2_0 = \deg \Ll$. One can contract $C_0$ to get a birational morphism $g: X \to Y$ with $Y$ a cone over the elliptic curve $E$. As $Y$ is rationally chain-connected, the MRCC fibration of $Y$ is a morphism to a point. Therefore, $\Aut^0(Y)$ is a linear algebraic group. As $\Aut^0(X) \to \Aut^0(Y)$ is an injective group homomorphism, $\Aut^0(X)$ is also a linear algebraic group. Therefore, $b: \Aut^0(X) \to \Aut^0(E)$ is the trivial homomorphism and $\Aut^0(X)_{\rm aff}=\Aut^0(X)$. By Lemma \ref{lem: lift vector field}, $f$ cannot be a locally constant fibration. In fact, $-K_X$ is not nef as $-K_X\cdot C_0 = \deg \Ll <0$.

Examples (1) and (2) can be found in \cite[Example 4.2.4]{BSU13}. \cite[Theorem 3]{Mar71} describes the automorphism group schemes of all ruled surfaces. The fact that $b$ is trivial in Example (3) can be directly obtained from \cite[Lemma 7]{Mar71}.
\end{example}

\bibliographystyle{alpha}
\bibliography{bibfile}
\end{document}